\documentclass[12pt]{article}
\usepackage[utf8]{inputenc}

\title{The Distinguishing Index of Mycielskian Graphs}

 \author{Rowan Kennedy \thanks{{\tt wilsonan3@mail.gvsu.edu}, Grand Valley State University, Allendale Charter Township, MI},  Lauren Keough  \thanks{{\tt keoulaur@gvsu.edu}, Grand Valley State University, Allendale Charter Township, MI}, Mallory Price \thanks{{\tt pricemal@gvsu.edu}, Grand Valley State University, Allendale Charter Township, MI}, Nick Simmons,  \thanks{{\tt simmonni@mail.gvsu.edu}, Grand Valley State University, Allendale Charter Township, MI}, Sarah Zaske \thanks{{\tt zaskes@mail.gvsu.edu}, Grand Valley State University, Allendale Charter Township, MI},  }
 \date{}

\usepackage{hyperref}
\usepackage[margin=.8in]{geometry}
\usepackage{amsmath,amsthm,amssymb}
\usepackage{tikz}
\usepackage{todonotes}
\usepackage{soul}
\usepackage[utf8]{inputenc}
\usepackage[T1]{fontenc}
\usepackage{multicol}
\usepackage{verbatim}
\usepackage{pgfplots}

\theoremstyle{plain}
\newtheorem{thm}{Theorem}
\newtheorem{lem}[thm]{Lemma}
\newtheorem{cor}[thm]{Corollary}

\newtheorem{conj}[thm]{Conjecture}

\theoremstyle{definition}

\newcommand\dist{\operatorname{Dist}}
\newcommand\diste{\operatorname{Dist^{\prime}}}
\newcommand\cbar{\overline{c}}

\definecolor{newpurple}{RGB}{220,55,245}

\begin{document}

\maketitle

\begin{abstract}

 
The distinguishing index gives a measure of symmetry in a graph. Given a graph $G$ with no $K_2$ component, a distinguishing edge coloring is a coloring of the edges of $G$ such that no non-trivial automorphism preserves the edge coloring. The distinguishing index, denoted $\operatorname{Dist^{\prime}}(G)$, is the smallest number of colors needed for a distinguishing edge coloring. The Mycielskian of a graph $G$, denoted $\mu(G)$, is an extension of $G$ introduced by Mycielski in 1955. In 2020, Alikhani and Soltani conjectured a relationship between $\diste(G)$ and $\diste(\mu(G))$.  We prove that for all graphs $G$ with at least 3 vertices, no connected $K_2$ component, and at most one isolated vertex, $\operatorname{Dist^{\prime}}(\mu(G)) \le \operatorname{Dist^{\prime}}(G)$, exceeding their conjecture. We also prove analogous results about generalized Mycielskian graphs. Together with the work in 2022 of Boutin, Cockburn, Keough, Loeb, Perry, and Rombach this completes the conjecture of Alikhani and Soltani.
    
\end{abstract}

 


\section{Introduction}

The Mycielskian of a graph $G$, denoted $\mu(G)$ was first introduced by Jan Mycielski in 1955 to show that there exist triangle-free graphs of arbitrarily large chromatic numbers \cite{M55}. Since then, there have been many counting parameters, such as variations of chromatic numbers and domination numbers studied about Mycielskian graphs, e.g., \cite{FiMcBo1998,LWLG2006,ChXi2006,BR2008,SV2023,BCKLPR24}. 
For a graph $G$ with vertices $v_1, \dots , v_n$, the \emph{Mycielskian of G}, denoted $\mu(G)$, has vertices $v_1, \dots, v_n, u_1, \dots, u_n, w$. For each edge $v_iv_j$ in $G$, the graph $\mu(G)$ has edges $v_iv_j, v_iu_j$, and $u_iv_j$. For each $u_i$ with $1\le i \le n$, $\mu(G)$ has edge $u_iw$. We call $v_1, \dots , v_n$ the \emph{original vertices}, $u_1, \dots , u_n$ the \emph{shadow vertices}, and $w$ the \emph{root}. The Mycielskian of $K_{1,3}$, $\mu(K_{1,3})$ is shown in Figure~\ref{fig:muK13}.

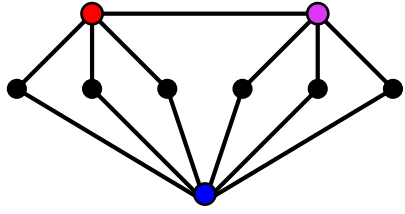
\begin{figure}
    \centering
     \begin{tikzpicture}[style=ultra thick, scale=1.0]
    \draw (1.5,2)--(-1.5,2);
    
    \draw(1.5,2)--(.5,1);
    \draw(0,-0.5)--(.5,1);
    \draw(1.5,2)--(1.5,1)--(0,-0.5);
    \draw(1.5,2)--(2.5,1)--(0,-0.5);

    \filldraw (.5,1) circle (3pt);
    \filldraw (1.5,1) circle (3pt);
    \filldraw (2.5,1) circle (3pt);
    
    \draw (-.5,1)--(0,-0.5);
    \draw(-1.5,2)--(-.5,1);
    \draw (-1.5,2)--(-1.5,1)--(0,-0.5);
    \draw (-1.5,2)--(-2.5,1)--(0,-0.5);
    \filldraw (-.5,1) circle (3pt);
    \filldraw (-1.5,1) circle (3pt);
    \filldraw (-2.5,1) circle (3pt);

    \draw[fill=blue,line width=1pt] (0,-0.4) circle (4pt);
    \draw[fill=newpurple,line width=1pt] (1.5,2) circle (4pt);
    \draw[fill=red,line width=1pt] (-1.5,2) circle (4pt);
\end{tikzpicture} \caption{The Mycielskian of $K_{1,3}$, $\mu(K_{1,3})$. The blue vertex is the unique vertex of maximum degree in the original graph, the pink vertex is its shadow, and the red vertex is the root.}
     \label{fig:muK13}
\end{figure}

The distinguishing number and distinguishing index, which are ways of measuring symmetries in a graph, are two counting parameters that  have been studied in relation to Mycielskian graphs \cite{AS2020, BCKLPR22}.  To define these terms we first need to define a graph automorphism. A \emph{graph automorphism} is a bijective function $\phi : V(G) \xrightarrow{} V(G)$ such that $x$ is adjacent to $y$ if and only if $\phi(x)$ is adjacent to $\phi(y)$ for all $x,y \in V(G)$. 
A \emph{distinguishing vertex coloring} of a graph $G$ is a coloring of the vertices of $G$ such that no non-trivial automorphism preserves the vertex coloring. The \emph{distinguishing number}, denoted $\dist(G)$ is the smallest number of colors needed for a distinguishing vertex coloring of $G$. Babai introduced this idea, calling it an asymmetric coloring, in \cite{B97}. Albertson and Collins independently introduced the same idea and the name distinguishing coloring in \cite{AC96}. 

Similarly, a \emph{distinguishing edge coloring} of a graph $G$ is a coloring of the edges of $G$ such that no non-trivial automorphism preserves the edge coloring and the  \emph{distinguishing index}, denoted $\diste(G)$ is the smallest number of colors needed for a distinguishing edge coloring of $G$.  The distinguishing index was defined by Kalinowski and Pil$\acute{\text{s}}$niak in \cite{KP15}. Note that graphs with a connected component that is $K_2$ do not have a distinguishing index, since any edge coloring of a $K_2$ has a nontrivial color-preserving automorphism. Similarly, graphs with more than one isolated vertex do not have a distinguishing index. Therefore, throughout the paper we assume all graphs $G$ do not have a $K_2$ component, and have at most one isolated vertex.

In 2020 Alikhani and Soltani proved that if $G$ has no vertices $u$ and $v$ such that $N(u)=N(v)$ then $\dist(\mu(G))\le \dist(\mu(G)) + 1$ whenever $G$ has at least $2$ vertices, and $\diste(\mu(G))\le \diste(\mu(G)) + 1$ whenever $G$ has at least $3$ vertices and no connected component $K_2$. They then proposed Conjecture~\ref{conj:AS}.

\begin{conj}\cite{AS2020}\label{conj:AS}
    Let $G$ be a connected graph of order $n \ge 3$. Then $\dist(\mu(G)) \le \dist(G)$ and $\diste(\mu(G)) \le \diste(G)$, except for a finite number of graphs.
\end{conj}

Notably, their conjecture does not require the graph to avoid vertices that have identical neighborhoods and tightens the inequality by removing the plus one. Boutin, Cockburn, Keough, Loeb, Perry, and Rombach proved and exceeded the distinguishing vertex coloring part of Conjecture~\ref{conj:AS} in 2022 \cite{BCKLPR22}. This paper completes the proof of the conjecture by showing the edge distinguishing inequality.

Since our proofs will rely on showing automorphisms are trivial based on facts about the graph and its coloring, we introduce definitions and notations about graphs.
Given a vertex $v$, the \emph{vertex neighborhood of $v$}, denoted $N_G(v)$, is the set of all vertices adjacent to $v$ in $G$. The \emph{degree of a vertex $v$} is $|N_G(v)|$ and is denoted $\deg_G(v)$. When the graph $G$ is clear, we omit the subscript and simply write $N(v)$ or $\deg(v)$.
The \emph{distance between two vertices $u$ and $v$ in $G$}, denoted $d(x,y)$ is the length of a shortest path between vertices $x$ and $y$.

By definition, automorphisms preserve adjacencies and non-adjacencies between vertices. As a result, automorphisms also preserve vertex degrees and distances between vertices.  



\section{Distinguishing Index of Mycielskian Graphs} \label{sec:mycielskian}

In a graph $G$, $u$ and $v$ are twins if $N(u)=N(v)$. For any graph $G$ that contains twins $u$ and $v$, there is an automorphism $\phi$ of $G$ such that $\phi(u)=v$ and $\phi(v)=u$ while $\phi(x)=x$ for all $x\in V(G)\setminus \{u,v\}$. In a vertex distinguishing coloring of $G$, this implies twins must receive different colors.
In an edge distinguishing coloring of $G$ with twins, we must prevent the automorphism that switches twins using a coloring of the edges. 
In \cite{AS2020} Alikhani and Soltani consider graphs without twins to avoid this complication. 
Lemma~\ref{lem:twincolors} establishes a fact about distinguishing edge colorings of graphs with twins that will aid us in proving Conjecture~\ref{conj:AS}.

\begin{lem}\label{lem:twincolors}
   Let $G$ be a graph with vertices $x_1$ and $x_2$ such that $N(x_1) = N(x_2)$. If $c$ is a distinguishing edge coloring of $G$, then there exists $v \in N(x_1)=N(x_2)$ such that $c(vx_1)\neq c(vx_2)$.
\end{lem}

\begin{proof}
     Given a graph $G$, label the vertices of $V(G)$ such that for $x_1,x_2\in V(G)$, $N(x_1) = N(x_2)$. We will prove the contrapositive: that is, if for all $v\in N(x_1)=N(x_2)$ we have $c(vx_1)=c(vx_2)$, then $c$ is not a distinguishing edge coloring of $G$.
     We assume that for all $v\in N(x_1)=N(x_2)$ we have $c(vx_1)=c(vx_2)$. Let $\phi$ be the automorphism that swaps twins $x_1$ and $x_2$, so that $\phi(x_1)=x_2$, $\phi(x_2)=x_1$, and $\phi(x)=x$ for all other vertices. For all $v\in N(x_1) = N(x_2)$ we have $c(vx_1)=c(vx_2)$, and so $c(\phi(vx_1))=c(vx_2)=c(vx_1)$ and $c(\phi(vx_2)=c(vx_1)=c(vx_2)$. Since all other vertices are fixed, $\phi$ is a nontrivial color-preserving automorphism and $c$ is not a distinguishing edge coloring.
\end{proof}

The strategy for proving Conjecture~\ref{conj:AS} is to separate out cases where the root $w$ is fixed in any automorphism. In Figure~\ref{fig:muK13} we can see that there is an automorphism of $\mu(K_{1,3})$ such that $\phi(w) = u_0$, where $u_0$ is the shadow vertex of the unique vertex of maximum degree in $K_{1,3}$.  In \cite{BCKLPR22}, the authors prove that for any graph $G$ and for any automorphism $\phi$ of $\mu(G$) the root $w$ is fixed unless $G = K_{1,m}$ for some $m\ge 0$. Since we rely on this result for our proofs, we state it precisely below.

\begin{lem}\label{lem:autos} \cite{BCKLPR22}
    If there is an automorphism $\phi$ of $\mu(G)$ that takes the root $w$ to any other vertex, then $G = K_{1,m}$ for some $m\ge 0$. Additionally, if $G = K_{1,m}$ for $m\ge 2$ then $\phi(w)\in \{w,u_0\}$ where $u_0$ is the shadow vertex of the unique vertex of maximum degree in $G$.
\end{lem}

Note that $K_{1,0}$ and $K_{1,1} = K_2$ do not have any edge-distinguishing colorings. Thus, for the edge distinguishing problem, we only consider $K_{1,m}$ for $m\ge 2$. We call the graphs $K_{1,m}$ for $m\ge 2$ star graphs and consider them in Section~\ref{subsec:mycielskistars}. We consider all other graphs in Section~\ref{subsec:mycielskiother}.

\subsection{Star Graphs}\label{subsec:mycielskistars} 
In this section we will show that $\diste(\mu(K_{1,m})) \le \diste(K_{1,m})$ proving Conjecture~\ref{conj:AS} for star graphs. In fact, Theorem~\ref{thm:smallmuK1m} shows that $\diste(K_{1,m})$ and $\diste(\mu(K_{1,m}))$ are only equal when $m=2$ and can be arbitrarily far apart.

\begin{figure}
\begin{center}
\begin{tikzpicture}[style=very thick, scale=1.4]
    \draw[blue] (1.5,2)--(-1.5,2);
    
    \draw[blue] (1.5,2)--(.5,1);
    \draw[red](0,0)--(.5,1);
    \draw[blue] (1.5,2)--(1.5,1)--(0,0);
    \draw[red] (1.5,2)--(2.5,1)--(0,0);

    \filldraw (.5,1) circle (2pt) node [anchor= west] {$u_1$};
    \filldraw (1.5,1) circle (2pt) node [anchor= west] {$u_2$};
    \filldraw (2.5,1) circle (2pt) node [anchor= west] {$u_3$};
    
    \draw[blue] (-.5,1)--(0,0);
    \draw[red](-1.5,2)--(-.5,1);
    \draw[blue] (-1.5,2)--(-1.5,1)--(0,0);
    \draw[red] (-1.5,2)--(-2.5,1)--(0,0);
    \filldraw (-.5,1) circle (2pt) node [anchor= east] {$v_1$};
    \filldraw (-1.5,1) circle (2pt) node [anchor= east] {$v_2$};
    \filldraw (-2.5,1) circle (2pt) node [anchor= east] {$v_3$};

    \filldraw (0,0) circle (2pt) node [anchor=north] {$v_0$};
    \filldraw (1.5,2) circle (2pt) node [anchor=west] {$w$};
    \filldraw (-1.5,2) circle (2pt) node [anchor=east] {$u_0$};
\end{tikzpicture}\\
\end{center}
    \caption{A distinguishing edge coloring on $\mu(K_{1,3})$ as described in the proof of Theorem~\ref{thm:smallmuK1m}.}
    \label{fig:muK13colored}
\end{figure}
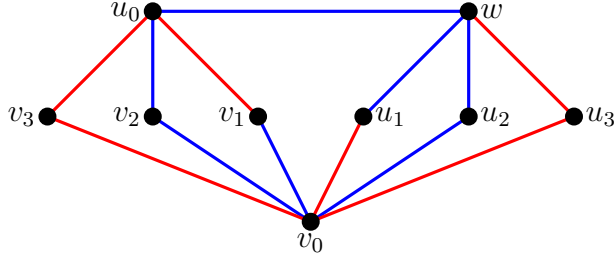


\begin{thm}\label{thm:smallmuK1m}
    For $m\geq 2$ 
    \[\diste\left(\mu(K_{1,m}) \right)=r\] where $r$ is the minimum natural number such that $r^2\ge m+1$. In particular,
    \[ \diste\left(\mu(K_{1,m})\right) \le m = \diste(K_{1,m}).\]
\end{thm}
\begin{proof}
    Let $V(K_{1,m})=\{v_0, v_1, v_2,\dots, v_m\}$ be the vertex set of $K_{1,m}$ such that $\deg(v_0)=m$. Let $u_0, u_1, u_2,\dots, u_m$ be the shadow vertices in $\mu(K_{1,m})$, and let $w$ be the root. Additionally, let $I=\{1,2,\dots, m\}$. Then we have the vertex set \[V(\mu(K_{1,m}))=\{v_0, v_1,\dots, v_m, u_0, u_1, \dots, u_m, w\}\] and the edge set 
    \[E(\mu(K_{1,m}))=\{v_0v_i,v_0u_i,u_0v_i,wu_i,wu_0 :  i\in I\}.\]
    We consider two important subsets of the edges - let $L_1$ and $L_2$ be sets of ordered pairs of edges such that $L_1=\{(u_0v_i,v_iv_0)  :  i\in I\}$ and $L_2=\{(wu_i,u_iv_0)  :  i\in I\}$. Note that each edge in $E(\mu(K_{1,m}))$ is in an ordered pair in $L_1$ or $L_2$ or is the edge $u_0w$. In Figure~\ref{fig:muK13colored}, $L_1$ represents the "elbows" drawn on the left, and $L_2$ represents the "elbows" drawn on the right.
    
    Choose $r$ least such that $r^2\geq m+1$. We will color $\mu(K_{1,m})$ with $r$ colors. Given $r$ colors, there are $r^2$ ordered pairs of colors. Let $P=\{p_1,p_2,\dots p_{r^2}\}$ be the set of all ordered pairs of $r$ colors. 
    First, we color the edges in $L_1$. For $i\in I$, we assign the color pair $p_i$ to the ordered pair $(u_0v_i, v_iv_0)$ so that $u_0v_i$ gets the first color in the ordered pair $p_i$ and $v_iv_0$ gets the second color in the ordered pair $p_i$. 
    Similarly, for the $i^{th}$ ordered pair in $L_2$, $(wu_i, u_iv_0)$, we assign the color pair $p_{i+1}$. Since $r^2\geq m+1$, there are enough color pairs to do so. 
    Finally, color the edge $u_0w$ any of the $r$ colors.

    We will now show this coloring is distinguishing. Suppose $\phi$ is a color-preserving automorphism on $\mu(K_{1,m})$. 
    By Lemma~\ref{lem:autos}, $\phi(w) \in \{w,u_0\}$. 
    By construction, every pair $p_j$ for $1\le j \le r^2$ is distinct so $\phi(v_i)\ne v_\alpha$ for any $\alpha\ne i$ and $\phi(u_i)\ne u_\alpha$ for any $\alpha\ne i$. Because the color pair $p_1$ colors an edge pair in $L_1$ and does not color any edge pair in $L_2$, there is no $\beta\in I$ such that $\phi(v_1)=u_\beta$. So $\phi(v_i)=v_i$ and $\phi(u_i)=u_i$ for every $i\in I$. It follows that $\phi(w)=w$ and $\phi(u_0)=u_0$. So we have a distinguishing edge coloring for $\mu(K_{1,m})$ with $r$ colors and $\diste(\mu(K_{1,m}))\le r$.

    To show that $\diste(\mu(K_{1,m})) = r$, suppose that we have a distinguishing coloring with $s$ colors such that $s<r$. Let  $P'=\{p_1,p_2,\dots p_{s^2}\}$ be the set of all ordered pairs of $s$ colors. Since $s<r$, and $r$ is least such that $r^2\ge m+1$, $s^2<m+1$ and  $P'$ contains at most $m$ color pairs. Note that for every $i\in I$, $N(v_i) = \{u_0, v_0\}$ and hence for all $i,j\in I$, $v_i$ and $v_j$ are twins. By Lemma~\ref{lem:twincolors}, every edge pair 
    in $L_1$ must be colored with a different color pair. 
    Since there are $m$ edge pairs in $L_1$ we must have $s^2\ge m$. Since $s^2\le m$ by assumption, suppose $s^2=m$. In this case all color pairs must be used to color the edge pairs of $L_1$ and all color pairs must be used to color the edge pairs of $L_2$. Because $L_1$ and $L_2$  use all of the color pairs of $P'$, we can construct a non-trivial automorphism $\phi$ that is color preserving where $\phi(w) = u_0$. So we have found that we cannot make a distinguishing edge coloring with fewer than $r$ colors and $\diste(\mu(K_{1,m})) = r$.
    


    Note that $\dist'(K_{1,m})=m$ for if any two edges were the same color, their degree $1$ vertices could be switched in a color preserving automorphism. For $m\ge 2$, we know $m^2\ge m+1$ and hence $r=m$ satisfies $r^2\ge m+1$. Since $r$ is least such that $r^2\ge m+1$, $r\le m$ and we conclude 
     \[ \diste\left(\mu(K_{1,m})\right) = r \le m = \diste(K_{1,m}).\]\end{proof}

\subsection{Not Star Graphs}\label{subsec:mycielskiother}

In this section we complete the proof of Conjecture~\ref{conj:AS} by proving $\diste(\mu(G))\le \diste(G)$ for all non-star graphs for which $\diste(G)$ is defined.

\begin{thm} \label{thm:notstars}
    Let $G$ be a graph with no connected $K_2$ component and at most one isolated vertex $G\neq K_{1,m}$ for any $m$. Then $\diste(\mu(G))\le \diste(G)$.
\end{thm}

\begin{proof}
Define the vertices of $\mu(G)$ to be $v_1,\dots, v_n$, their shadows to be $u_1,\dots, u_n$ and the root to be $w$.
Let $c$ be a distinguishing coloring of $E(G)$ with colors $1,\dots, q$. We define a coloring $\cbar$ of $E(\mu(G))$ that ``mimics" the coloring on $E(G)$. If  $xy \in\{ v_jv_k, v_ju_k, v_ku_j\}$ for any $1\le k,j\le n$ then $\cbar(xy) = c(v_jv_k)$. Lastly, $\cbar(wu_i) = 1$ for all $1\le i\le n$.

Let $\phi$ be a color-preserving automorphism. Since $G \neq K_{1,m}$ we know by Lemma~\ref{lem:autos} that for the root $w$ of $\mu(G)$, $\phi(w)=w$. For all $1\le i\le n$, $d(u_i, w) = 1$ and $d(v_i, w) = 2$. Therefore, for all $1\le i\le n$, $\phi(u_i) = u_j$ for some $1\le j\le n$ and $\phi(v_i) = v_j$ for some $1\le j\le n$. That is, the levels of $\mu(G)$ are fixed set-wise. Moreover, $\cbar$ fixes  $v_i\in V(\mu(G))$ for $1\le i \le n$ because $\cbar$ restricts to a distinguishing edge coloring of $G$. Let $u_k$ be the shadow of a non-twin vertex in $\mu(G)$. Then $u_k$ has a unique and fixed neighborhood, and $\phi(u_k)=u_k$ for all such $u_k$.

We will show that twins $u_i, u_j \in V(\mu(G))$ are fixed using a proof by contradiction. So, we assume that $u_i, u_j$ are not fixed by $\phi$, so $\phi(u_i) = u_j$ and $\phi(u_j) = u_i$. Since $u_i, u_j$ are twins, there are twins $v_i$ and $v_j$ in $G$ where $N(v_i) =N(v_j)$, so by Lemma \ref{lem:twincolors}, there exists some $v_0 \in N(v_i) = N(v_j)$ such that $c(v_kv_i) \neq c(v_0v_j)$. Since $c(v_kv_i) = \cbar(v_ku_i)$ and $c(v_kv_j) = \cbar(v_ku_j)$ we know $\cbar(v_ku_i)\ne  \cbar(v_ku_j)$. Then $\phi$ cannot be a color-preserving automorphism on $\mu(G)$ which is a contradiction. Thus $u_i$ and $u_j$ are fixed as desired.
\end{proof}

Theorems~\ref{thm:smallmuK1m} and \ref{thm:notstars} prove Conjecture~\ref{conj:AS}. In fact, we exceed the conjecture as the proofs do not require $G$ to be connected and there are no exceptions besides those for which $\diste(G)$ is undefined.

\begin{cor}\label{cor:muproof}
 For all graphs $G$ with $|V(G)|\ge 3$, no connected $K_2$ component, and at most one isolated vertex
    \[\diste(\mu(G))\le \diste(G).\]
\end{cor}

\section{Generalized Mycielskian Graphs}

The generalized Mycielskian, sometimes called the cone over a graph was introduced by Steibitz \cite{S85}. Let $G$ be a graph with vertices $v_1, \dots , v_n$. Then the \emph{generalized Mycielskian} of $G$ with $t$ layers, denoted $\mu_t(G)$, has vertex set 
\[ V(\mu_t(G)) = \{u_1^0, \dots , u_n^0, u_1^1, \dots , u_n^1, \dots , u_1^t, \dots , u_n^t, w\}.\] 
The vertices $u_1^0, \dots , u_n^0$ represent the "original" vertices $v_1, \dots , v_n$ which can be thought of as the zeroth layer. 
For each edge $v_iv_k$ in $G$ with $1\le i,k\le n$, $\mu_t(G)$ has edges $u_i^0u_k^0$, $u_i^ju_k^{j+1}$, and $u_k^ju_i^{j+1}$ for $0\leq j < t$. Lastly, $\mu_t(G)$  has the edges $u_i^tw$ for $1\le i \le n$. In Figure~\ref{fig:genmyc}, $\mu(K_{1,2})$, $\mu_2(K_{1,2})$ and $\mu_3(K_{1,2})$ are shown.

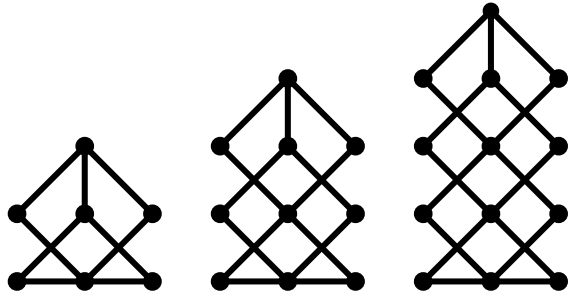
\begin{figure}
    \centering
    \begin{tikzpicture}[line width=0.8mm,scale=.3]

\draw (0,0) -- (3,0);
\draw (3,0)--(6,0);


\draw (0,0) -- (3,3)--(0,6)--(3,9)--(3,12);
\draw  (3,0) -- (0,3)--(3,6)--(0,9)--(3,12)--(6,9);
\draw(6,0) -- (3,3)--(6,6)--(3,9);
\draw (3,0) -- (6,3)--(3,6)--(6,9);

\draw[fill=black!100, line width=1mm] (3,12) circle (.2);

\foreach \x in {0,1.5,3}
	{\draw[fill=black!100]  (2*\x,0) circle (.28);}

\foreach \x in {0,1.5,3}
	{\draw[fill=black!100]  (2*\x,3) circle (.28);}	
 
\foreach \x in {0,1.5,3}
	{\draw[fill=black!100]  (2*\x,6) circle (.28);}	
 
 \foreach \x in {0,1.5,3}
	{\draw[fill=black!100]  (2*\x,9) circle (.28);}	

 \begin{scope}[shift={(-9,0)}]
        \draw (0,0) -- (6,0);
        \draw[black](3,0)--(0,3)--(3,6)--(3,9);
        \draw[black] (0,0) -- (3,3)--(0,6)--(3,9);
        \draw(6,0) -- (3,3)--(6,6)--(3,9);
        \draw (3,0) -- (6,3)--(3,6);
        \draw[fill=black!100]  (3,9) circle (.28);
        \foreach \x in {0,1.5,3}
	   {\draw[fill=black!100]  (2*\x,6) circle (.28);}
        \foreach \x in {0,1.5,3}
	   {\draw[fill=black!100]  (2*\x,3) circle (.28);}
    \foreach \x in {0,1.5,3}
	   {\draw[fill=black!100]  (2*\x,0) circle (.28);}
    \begin{scope}[shift={(-9,0)}]
    \draw (0,0) -- (6,0);
        \draw[black](3,0)--(0,3)--(3,6);
        \draw[black] (0,0)--(3,3)--(3,6);
        \draw(6,0)--(3,3);
        \draw (3,0)--(6,3)--(3,6);
        \draw[fill=black!100]  (3,6) circle (.28);
        \foreach \x in {0,1.5,3}
	   {\draw[fill=black!100]  (2*\x,3) circle (.28);}
    \foreach \x in {0,1.5,3}
	   {\draw[fill=black!100]  (2*\x,0) circle (.28);}
    \end{scope}
 \end{scope}
 
\end{tikzpicture}
    \caption{From left to right, $\mu(K_{1,2})$, $\mu_2(K_{1,2})$ and $\mu_3(K_{1,2})$.}
    \label{fig:genmyc}
\end{figure}

In \cite{BCKLPR22} it is shown that a result analogous to Lemma~\ref{lem:autos} holds for generalized Mycielskian graphs. We will rely on this result, so we repeat it here, insisting $m\ge 2$.

\begin{lem}\label{lem:autost}\cite{BCKLPR22} Let $G$ be a graph, let $t\ge 1$, and let $\phi$ be an automorphism of $\mu_t(G)$. If $G\neq K_{1,m}$ for any $m\ge 2$ then, for the root vertex $w$, we have $\phi(w)=w$.  If $G = K_{1,m}$ for $m\ge 2$ then $\phi(w)\in \{ w, u_0^t\}$ where $u_0^t$ is the top-level shadow vertex of the vertex of degree $m$ in $K_{1,m}$.
\end{lem}

Motivated by Lemma~\ref{lem:autost}, we again consider star graphs separately.


\subsection{Star Graphs}

By Lemma \ref{lem:autost}, the star graphs $K_{1,m}$ for $m\geq 2$ have an automorphism where, for the root $w$, $\phi(w)\neq w$. In particular, for the root $w$, $\phi(w)\in \{w,u_0^t\}$. In other words, $w$ is fixed or is mapped to the shadow vertex of the unique vertex of maximum degree in $G$ on level $t$. The automorphisms that do not fix the root $w$ are seen in the horizontal symmetry shown in Figure~\ref{fig:kite}.

\begin{thm}\label{thm:mutstars}
    Let $m\geq 2$ and $t\in \mathbb{N}$.  Then 
    \[\diste\left(\mu_t(K_{1,m})\right)=r\]
    where $r$ is the minimum natural number such that $r^2\ge m+1$. In particular, for $m\ge 2$. \[\diste(\mu_t(K_{1,m})) \le \diste(K_{1,m}).\] 
\end{thm}
\begin{proof}
     Let $V(\mu_t((K_{1,m}))=\{u^j_0, u^j_1, u^j_2,\dots, u^j_m: 0\le j\le t\}\cup \{w\}$ where $u^0_0$ is the center of $K_{1,m}$ and $w$ is the root.
    
    
To define the edge set, we start by defining $L_i$ for $0\le i\le t$ as sets of ordered pairs of edges. Let $L_0=\{(u^1_0u^0_i,u^0_iu^0_0)  :  i\in I\}$,
$L_{t}=\{(wu^{t}_0,u^{t}_iu^{t-1}_0) :  i\in I\}$, and 
\[L_\alpha=\{(u^{\alpha+1}_0u^{\alpha}_i,u^{\alpha}_iu^{\alpha-1}_0) :  1\le i \le n\}\] 
where $1\le \alpha < t$. 
The edges in ordered pairs in each $L_\alpha$ plus $u_0^tw$ make up $E(\mu_t(K_{1,m}))$. In Figure~\ref{fig:kite} the sets $L_i$ for $0\le i \le 5$ are shown for $\mu_5(K_{1,3})$.
     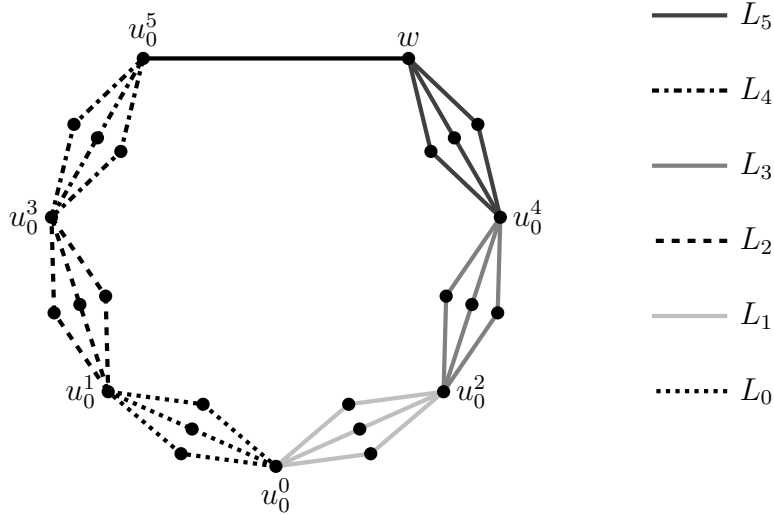
\begin{figure}
     \centering
     \begin{tikzpicture}[style=ultra thick,scale=1]
     \draw(126:3)--(54:3);
    \draw[white] (-6,3)--(-5,3) [];
    \fill[white] (-6,0) node [anchor=east] {$L_1$};

    \draw[darkgray] (6,3)--(5,3);
    \fill[] (6,3) node [anchor=west] {$L_5$};

    \draw[gray] (6,1)--(5,1);
     \fill[] (6,1) node [anchor=west] {$L_3$};
    
    \draw[lightgray] (6,-1)--(5,-1);
     \fill[] (6,-1) node [anchor=west] {$L_1$};
     
    \draw[dashed] (6,0)--(5,0);
     \fill[] (6,0) node [anchor=west] {$L_2$};
    
    \draw[dotted] (6,-2)--(5,-2);
     \fill[] (6,-2) node [anchor=west] {$L_0$};
    
    \draw[dashdotted] (6,2)--(5,2);
     \fill[] (6,2) node [anchor=west] {$L_4$};

    \draw[darkgray](270+24*6:3)--(270+24*5:3.1)--(270+24*4:3);
    \draw[darkgray](270+24*6:3)--(270+24*5:2.38)--(270+24*4:3);
    \draw[darkgray](270+24*6:3)--(270+24*4:3);

    \draw[gray](270+24*4:3)--(270+24*3:3.1)--(270+24*2:3);
    \draw[gray](270+24*4:3)--(270+24*3:2.38)--(270+24*2:3);
    \draw[gray](270+24*4:3)--(270+24*2:3);
    
    \draw[lightgray](270+24*2:3)--(270+24:3.1)--(270:3);
    \draw[lightgray](270+24*2:3)--(270+24*1:2.38)--(270:3);
    \draw[lightgray](270+24*2:3)--(270:3);

    \draw[dotted][black](270:3)--(270-24:3.1)--(270-24*2:3);
    \draw[dotted](270-24*2:3)--(270-24*1:2.38)--(270:3);
    \draw[dotted](270:3)--(270-24*2:3);
        
    \draw[dashed](270-24*2:3)--(270-24*3:3.1)--(270-24*4:3);
    \draw[dashed](270-24*4:3)--(270-24*3:2.38)--(270-24*2:3);
    \draw[dashed](270-24*2:3)--(270-24*4:3);

    \draw[dashdotted](270-24*4:3)--(270-24*5:3.1)--(270-24*6:3);
    \draw[dashdotted](270-24*6:3)--(270-24*5:2.38)--(270-24*4:3);
    \draw[dashdotted](270-24*4:3)--(270-24*6:3);

        
        
        
        \foreach \a in {-5,-3,-1,1,3,5}{
           \fill(270+24*\a:2.74) circle (2.5pt);
           \fill(270+24*\a:3.1) circle (2.5pt);
           \fill(270+24*\a:2.38) circle (2.5pt);
           }
        
    \fill(126:3) circle (2.5pt) node [anchor=south] {$u_0^{5}$};
    \fill(126+24*2:3) circle (2.5pt) node [anchor= east] {$u_0^3$};
    \fill(126+24*4:3) circle (2.5pt) node [anchor= east] {$u_0^1$};
    
    \fill(270:3) circle (2.5pt) node [anchor=north] {$u_0^0$};
    
    \fill(54-24*4:3) circle (2.5pt) node [anchor= west] {$u_0^2$};
    \fill(54-24*2:3) circle (2.5pt) node [anchor= west] {$u_0^4$};
    \fill(54:3) circle (2.5pt) node [anchor=south] {$w$};


    \end{tikzpicture}
     \caption{The sets $L_i$ for $0\le i\le 5$ in $\mu_5(K_{1,3})$.}
         \label{fig:kite}
     \end{figure}
    
    Choose $r$ least such that $r^2\geq m+1$. Given $r$ colors, there are $r^2$ ordered pairs of colors. Let $P=\{p_1,p_2,\dots p_{r^2}\}$ be the set of all ordered pairs of $r$ colors.
    For each $L_\alpha$ with $1\le \alpha <t$, we assign the color pair $p_i$ to the $i^{th}$ ordered pair  in $L_\alpha$, $(u^{\alpha+1}_0u^{\alpha}_i,u^{\alpha}_iu^{\alpha-1}_0)$ so that $u^{\alpha+1}_0u^{\alpha}_i$ gets the first color in the ordered pair $p_i$ and $u^{\alpha}_iu^{\alpha-1}_0$ gets the second color in the ordered pair $p_i$. Similarly, for the $i^{th}$ ordered pair in $L_0$, $(u_0^1u^0_i, u^{0}_iu^{0}_0)$, we assign the color pair $p_{i}$. 
    Finally, for the $i^{th}$ ordered pair in $L_t$, $(wu^t_i, u^{t}_iu^{t-1}_0)$, we assign the color pair $p_{i+1}$.   
    Since $r^2\geq m+1$, there are enough color pairs to do so. Finally, color the edge $u_0w$ any of the $r$ colors. Examples of this coloring on $\mu_5(K_{1,3})$ and $\mu_2(K_{1,5})$ are shown in Figure~\ref{fig:mu5K13colored}.
    
    We'll now prove this coloring is distinguishing. Suppose $\phi$ is a color-preserving automorphism on $\mu_{t}(K_{1,m})$ and suppose $\phi(w) = u_0^t$. Since $u_m^t$ has degree $2$ and is distance $1$ from $w$, $\phi(u_m^t)$ must have degree $2$ and be distance $1$ from $\phi(w) = u_0^t$. Thus $\phi(u_m^t) = u_i^{t-1}$ for some $1\le i\le t$. Similarly, by distance and degree, $\phi(u_0^{t-1}) = u_0^{t-2}$. Note that $(wu_m^{t}, u_m^tu_0^{t-1}) \in L_t$ has color pair $p_{m+1}$, and no pair of edges in $L_{t-1}=\{(u^{t}_0u^{t-1}_i,u^{t-1}_iu^{t-2}_0) :  1\le i \le n\}$ has this color pair. So $u_m^t$ can't map to $u_i^{t-1}$ because then the pair of edges $(wu_m^{t}, u_m^tu_0^{t-1})$ would map to the pair of edges $(u_0^tu_i^{t-1}, u_i^{t-1}u_0^{t-2})$ and this wouldn't be color-preserving. So $\phi(w) \neq u_0^t$. 
    Therefore, by Lemma~\ref{lem:autost}, $\phi(w)=w$. 
    
    Since $w$ is fixed, by distance and degree, the levels of $\mu_t(K_{1,m})$ are fixed setwise by $\phi$. The construction of our coloring means that on a given level $j$ with $0\le j\le t$, we used a distinct color ordered pair for the edges incident to any degree $2$ vertex. So, $\phi(u^j_i)\ne u^j_k$ for any $k\ne i$. Thus, $\phi(u^j_i)=u^j_i$ for $0\le j \le t$ and $1\le i\le m$. Since we also know $\phi(u^j_0)=u^j_0$ for $0\le j \le t$ and $\phi(w)=w$, we conclude that $\phi$ is the trivial automorphism and $\diste(\mu_t(K_{1,m}))\le r$.

    To show that $\diste(\mu(K_{1,m})) = r$, we assume for the sake of a contradiction that there's a distinguishing coloring with $s$ colors such that $s<r$. Let  $P'=\{p_1,p_2,\dots p_{s^2}\}$ be the set of all ordered pairs of $s$ colors. Since $s<r$, and $r$ is least such that $r^2\ge m+1$, we know a$s^2\le m$.
    
    First we'll show $s^2\ge m$. 
   For each $1\le j\le t$ and $1\le i \le n$,  $N(u^j_i) = \{u^{j+1}_0, u^{j-1}_0\}$.  Similarly, for $1\le i\le n$, $N(u^0_i) = \{u^{1}_0, u^{0}_0\}$ and $N(u^t_i) = \{w, u^{t-1}_0\}$. Hence, for a fixed $j$, and all $1\le a,b \le n$, $u_a^j$ and $u_b^j$ are twins. By Lemma~\ref{lem:twincolors}, 
    every edge pair in $L_j$ must be colored with a different color pair. There are $m$ edge pairs in $L_t$ and $s^2$ color pairs, so $s^2\ge m$. 

    Suppose $s^2=m$. In this case all color pairs must be used to every color the edge pairs of each $L_j$ for $0\le j \le t$. 
    Let $\mathcal{L}_j$ be the set of vertices $\{u^j_i \mid 1\le i\le n\}$. for every $j$, each vertex in $\mathcal{L}_j$ is the same distance from $w$ so we define $d(\mathcal{L}_j, w)$ to be this distance.  We will consider the cases for $t$ even and $t$ odd separately.
    
    When $t$ is odd, for each $\mathcal{L}_j$ that is distance $d$ from $w$ there is exactly one other set, call it $\mathcal{L}_j'$, that is distance $d$ from $u^t_0$.  
    Let $\psi$ be the automorphism which swaps $w$ and $u^t_0$ and maps the $i^{th}$ vertex in $\mathcal{L}_j$ to the vertex in $\mathcal{L}_j'$ that has the appropriate colors on its incident edges and complete the automorphism so that distances are preserved. Then $\psi$ is a color preserving automorphism. 
    
    If $t$ is even, then, as seen in Figure~\ref{fig:mu5K13colored}, $\mathcal{L}_0=\mathcal{L}_{0}'$.  Since all $m$ colorings were used, for any edge pair $\epsilon$ in $L_0$ with an ordered pair of different colors, the reverse ordered pair is also assigned to an edge pair $\delta$ in $L_0$. Let $\psi$ be the automorphism that swaps $w$ and $u^t_0$ and for $j\ge 1$, maps the $i^{th}$ vertex in $\mathcal{L}_j$ to the vertex in $\mathcal{L}_j'$ that has the appropriate colors on its incident edges. Let $\psi'$ be the automorphism which swaps the vertices in $\mathcal{L}_0$ such that $\epsilon$ maps to $\delta$ and vice versa. Then $\psi'\psi$ is a color preserving automorphism. 

    So we have shown that we can always construct a non-trivial automorphism $\phi$ that is color preserving when $s^2=m$. Therefore, we cannot make a distinguishing edge coloring with fewer than $r$ colors and $\diste(\mu(K_{1,m})) = r$.
    
     For $m\ge 2$, we know $m^2\ge m+1$ and hence $r=m$ satisfies $r^2\ge m+1$. Since $r$ is least such that $r^2\ge m+1$, $r\le m$ and we conclude that 
     \[ \diste\left(\mu_t(K_{1,m})\right) \le r \le m = \diste(K_{1,m}).\]
\end{proof}

\begin{figure}
    \centering
    \begin{tikzpicture}[style=ultra thick,scale=1.4]
     \draw[blue](126:2)--(54:2);

    \draw[red](270+24*6:2)--(270+24*5:2.22)--(270+24*4:2)--(270+24*3:2.22)--(270+24*2:2)--(270+24:2.22)--(270:2);
    
    \draw[red](270:2)--(270-24:2.22)--(270-24*2:2)--(270-24*3:2.22)--(270-24*4:2)--(270-24*5:2.22)--(270-24*6:2);
        
    \draw[red](270+24*6:2)--(270+24*5:1.42);
    \draw[blue](270+24*5:1.42)--(270+24*4:2);
    \draw[blue](270+24*4:2)--(270+24*3:1.42);
    \draw[red](270+24*3:1.42)--(270+24*2:2);
    \draw[blue](270+24*2:2)--(270+24*1:1.42);
    \draw[red](270+24*1:1.42)--(270:2);

    \draw[blue](270-24*6:2)--(270-24*5:1.42);
    \draw[red](270-24*5:1.42)--(270-24*4:2);
    \draw[blue](270-24*4:2)--(270-24*3:1.42);
    \draw[red](270-24*3:1.42)--(270-24*2:2);
    \draw[blue](270-24*2:2)--(270-24*1:1.42);
    \draw[red](270-24*1:1.42)--(270:2);
        
    \draw[blue](270+24*6:2)--(270+24*4:2)--(270+24*2:2)--(270:2)--(270-24*2:2)--(270-24*4:2)--(270-24*6:2);

        \foreach \a in {-5,-3,-1,1,3,5}{
           \fill(270+24*\a:1.83) circle (2.3pt);
           \fill(270+24*\a:2.22) circle (2.3pt);
           \fill(270+24*\a:1.43) circle (2.3pt);
           }
        
    \fill(126:2) circle (2.3pt) node [anchor=south] {$u_0^{5}$};
    \fill(126+24*2:2) circle (2.3pt) node [anchor= east] {$u_0^3$};
    \fill(126+24*4:2) circle (2.3pt) node [anchor= east] {$u_0^1$};
    
    \fill(270:2) circle (2.3pt) node [anchor=north] {$u_0^0$};
    
    \fill(54-24*4:2) circle (2.3pt) node [anchor= west] {$u_0^2$};
    \fill(54-24*2:2) circle (2.3pt) node [anchor= west] {$u_0^4$};
    \fill(54:2) circle (2.3pt) node [anchor=south] {$w$};
    \begin{scope}[shift={(6,0.42)}]
    \draw[black] (1.4,1.2)--(-1.4,1.2);
    
    \draw[black] (1.4,1.2)--(0:0.6);
    \draw[black] (1.4,-1.2)--(0:0.6);

    \draw[][blue] (1.4,1.2)--(0:0.6+0.4*1);
    \draw[][blue] (1.4,-1.2)--(0:0.6+0.4*1);

    \draw[][blue] (1.4,1.2)--(0:0.6+0.4*2);
    \draw[black] (1.4,-1.2)--(0:0.6+0.4*2);
    
    \draw[black] (1.4,1.2)--(0:0.6+0.4*3);
    \draw[][blue] (1.4,-1.2)--(0:0.6+0.4*3);
    
    \draw[black] (1.4,1.2)--(0:0.6+0.4*4);
    \draw[red] (1.4,-1.2)--(0:0.6+0.4*4);

    \filldraw (0:0.6) circle (1.8pt) node [anchor= west] {};
    \filldraw (0:0.6+0.4*1) circle (1.8pt) node [anchor= west] {};
    \filldraw (0:0.6+0.4*2) circle (1.8pt) node [anchor= west] {};
    \filldraw (0:0.6+0.4*3) circle (1.8pt) node [anchor= west] {};
    \filldraw (0:0.6+0.4*4) circle (1.8pt) node [anchor= west] {};

    \draw[black] (-1.4,1.2)--(180:0.6);
    \draw[black] (-1.4,-1.2)--(180:0.6);

    \draw[][blue] (-1.4,1.2)--(180:0.6+0.4*1);
    \draw[][blue] (-1.4,-1.2)--(180:0.6+0.4*1);

    \draw[][blue] (-1.4,1.2)--(180:0.6+0.4*2);
    \draw[black] (-1.4,-1.2)--(180:0.6+0.4*2);
    
    \draw[black] (-1.4,1.2)--(180:0.6+0.4*3);
    \draw[][blue] (-1.4,-1.2)--(180:0.6+0.4*3);
    
    \draw[red] (-1.4,1.2)--(180:0.6+0.4*4);
    \draw[red] (-1.4,-1.2)--(180:0.6+0.4*4);
   
    \filldraw (180:0.6) circle (1.8pt) node [anchor= west] {};
    \filldraw (180:0.6+0.4*1) circle (1.8pt) node [anchor= west] {};
    \filldraw (180:0.6+0.4*2) circle (1.8pt) node [anchor= west] {};
    \filldraw (180:0.6+0.4*3) circle (1.8pt) node [anchor= west] {};
    \filldraw (180:0.6+0.4*4) circle (1.8pt) node [anchor= west] {};

    \draw[black] (-1.4,-1.2)--(270:0.6);
    \draw[black] (1.4,-1.2)--(270:0.6);

    \draw[][blue] (-1.4,-1.2)--(270:0.6+0.4*1);
    \draw[][blue] (1.4,-1.2)--(270:0.6+0.4*1);

    \draw[][blue] (-1.4,-1.2)--(270:0.6+0.4*2);
    \draw[black] (1.4,-1.2)--(270:0.6+0.4*2);
    
    \draw[black] (-1.4,-1.2)--(270:0.6+0.4*3);
    \draw[][blue] (1.4,-1.2)--(270:0.6+0.4*3);
    
    \draw[red] (-1.4,-1.2)--(270:0.6+0.4*4);
    \draw[red] (1.4,-1.2)--(270:0.6+0.4*4);
    
    \filldraw (0:0.6) circle (1.8pt) node [anchor= west] {};
    \filldraw (0:0.6+0.4*1) circle (1.8pt) node [anchor= west] {};
    \filldraw (0:0.6+0.4*2) circle (1.8pt) node [anchor= west] {};
    \filldraw (0:0.6+0.4*3) circle (1.8pt) node [anchor= west] {};
    \filldraw (0:0.6+0.4*4) circle (1.8pt) node [anchor= west] {};

    \filldraw (-1.4,-1.2) circle (1.8pt) node [anchor=north] {$u^0_0$};
    \filldraw (1.4,-1.2) circle (1.8pt) node [anchor=north] {$u^1_0$};
    \filldraw (1.4,1.2) circle (1.8pt) node [anchor=west] {$w$};
    \filldraw (-1.4,1.2) circle (1.8pt) node [anchor=east] {$u^2_0$};

    \filldraw (270:0.6) circle (1.8pt) node [anchor= west] {};
    \filldraw (270:0.6+0.4*1) circle (1.8pt) node [anchor= west] {};
    \filldraw (270:0.6+0.4*2) circle (1.8pt) node [anchor= west] {};
    \filldraw (270:0.6+0.4*3) circle (1.8pt) node [anchor= west] {};
    \filldraw (270:0.6+0.4*4) circle (1.8pt) node [anchor= west] {};
    \end{scope}
    \end{tikzpicture}
    
    \caption{Distinguishing colorings of $\mu_5(K_{1,3})$ and $\mu_2(K_{1,5})$ using method described in proof of Theorem \ref{thm:mutstars}. On the left, $m=3$ so $r=2$ and the set $P = \{(\text{red},\text{red}), (\text{blue}, \text{blue}), (\text{red},\text{blue}), (\text{blue},\text{red})\}$. On the right, $m=5$ so $r=3$ and the set $P= \{(\text{red},\text{red}), (\text{blue}, \text{black}), (\text{black},\text{blue}), (\text{blue},\text{blue}), (\text{black}, \text{black})\}$.}
    \label{fig:mu5K13colored}
\end{figure}
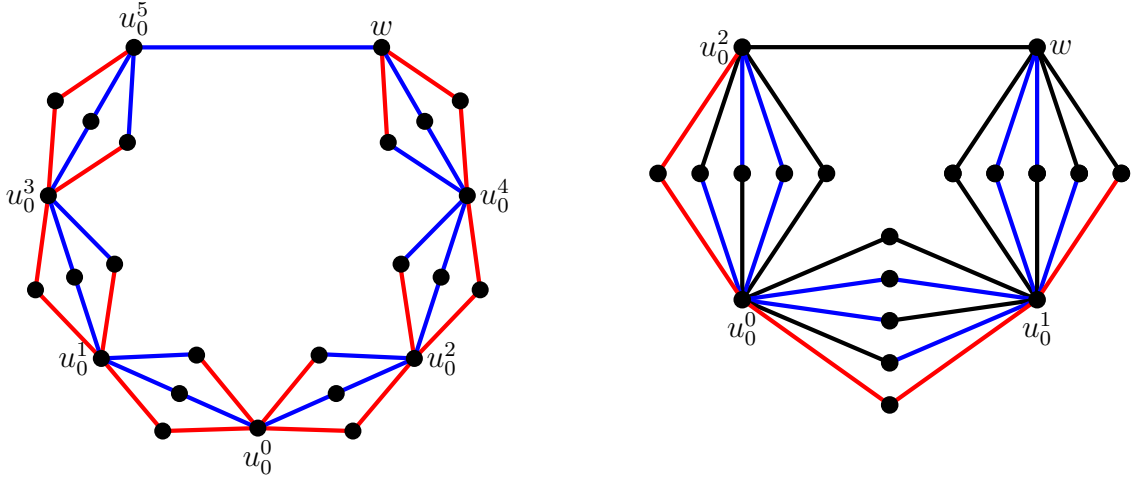
\subsection{Not Star Graphs}
In this section we will show that $\diste(\mu_t(G)) \le \diste(G)$ for all graphs that are not star graphs and for which $\diste(G)$ and $\diste(\mu_t(G))$ are defined. Note that when $t\ge 2$, an isolated vertex in $G$ will result in $2$ isolated vertices in $\mu_t(G)$. Since isolated vertices can be swapped in any color-preserving automorphism, $\diste(\mu_t(G))$ would be undefined. So in this section, we insist $G$ has no connected $K_2$ component and no isolated vertices. 

\begin{thm} \label{thm:notstarst} For any graph $G$ with $|V(G)| \ge 3$, no connected $K_2$ component, no isolated vertices, and $G\neq K_{1,m}$ for any $m$, then $\diste(\mu_t(G))\leq \diste(G)$.
\end{thm}
\begin{proof}
     Let $|V(G)| = n$ and define $V(\mu_t((G))=\{u^j_1, u^j_2,\dots, u^j_n: 0\le j\le t\}\cup \{w\}$ where $w$ is the root. Let $c$ be a distinguishing coloring of $E(G)$ with colors $1, \dots , q$. We define a coloring of $E(\mu_t(G))$ called $\cbar$. For an edge $u_i^0 u_k^0\in E(G)$, we let $\cbar(u_i^0u_k^0)=\cbar(u_i^ju_k^{j-1})=\cbar(u_k^ju_i^{j-1}) = c(u_i^0u_k^0)$ for $1 \le j \le t$. That is, we mimic the coloring $c$ so that edges between layers are colored with the same color as the edge that produced them.

    Let $\phi$ be a color-preserving automorphism. We will show $\phi$ is trivial. To start, since $G\neq K_{1,m}$ for any $m$ we know by Lemma~\ref{lem:autost} that $w$ is fixed. Assume for the sake of a contradiction that 
    $\phi(u_i^j)\ne u_i^j$ for some $1\le i\le n$ and $0\le j\le t$. Since $w$ is fixed each level of $\mu_t(G)$ is fixed setwise. Because the levels are fixed setwise and $\cbar$ restricts to a distinguishing edge coloring of $G$, $\phi(u_i^0) = u_i^0$ for $1\le i\le n$. That is, the vertices on the zeroth level are fixed.
    Without loss of generality, assume $j\ge 1$ is least such that a vertex on level $j$ is not fixed. That is, assume if $0\le \ell <j$ then $\phi(u_i^\ell) = u_i^\ell$ for all $1\le i \le n$.
    Each non-twin vertex on level $j$ in $\mu_t(G)$ has a unique neighborhood on level $j-1$ and that neighborhood is fixed by the minimality of $j$. Hence $u_i^j$ must be a twin vertex and must map to a twin on the same level.
    
     Suppose $\phi(u_i^j)=u_k^j$.  So $N(u_i^j) = N(u_k^j)$ and, by the construction of a Mycielskian, $N(u_k^0)= N(u_j^0)$. By Lemma \ref{lem:twincolors} we know that there exists some $u_\ell^0 \in N(u_i^0)= N(u_k^0)$ such that $c(u_i^0u_\ell^0) \ne c(u_k^0u_\ell^0)$. Moreover, because of the choice of coloring we know $\cbar(u_\ell^{j-1}u_i^j) \ne \cbar(u_\ell^{j-1}u_k^j)$. Since vertices on level $j-1$ are fixed we cannot have $\phi(u_i^j)=u_k^j$. Thus $\phi$ must be the trivial automorphism and there exists a distinguishing coloring of $\mu_t(G)$ with $\diste(G)$ colors. So, $\diste(\mu_t(G)) \le \diste(G)$.
\end{proof}

Theorems~\ref{thm:mutstars} and \ref{thm:notstarst} show that Conjecture~\ref{conj:AS} remains true for generalized Mycielskians. We state this result in the following theorem.

\begin{thm}\label{thm:mutproof}
For all graphs $G$ with $|V(G)|\ge 3$, no isolated vertices, and no connected $K_2$ component and for all $t\in\mathbb{N}$
    \[\diste(\mu_t(G))\le \diste(G).\]
\end{thm}

As mentioned in the introduction, the Mycielskian of a graph was first introduced to show that there exist triangle-free graphs of arbitrarily large chromatic numbers. To do so, Mycielski showed that the construction increases the chromatic number while preserving the property of being triangle-free. Iterating the construction on $K_2$ gives the result. We use $\mu^p(G)$ to denote the Mycielskian of $G$ iterated $p$ times and $\mu_t^p(G)$ to represent the generalized Mycielskian of $G$ with $t$ levels, iterated $p$ times. More precisely, $\mu_t^1(G) = \mu_t(G)$ and $\mu_t^{p}(G) = \mu_t(\mu_t^{p-1}(G))$.  Corollary~\ref{cor:iterated} follows from Theorem~\ref{thm:mutproof}.

\begin{cor}\label{cor:iterated} For all graphs $G$ with $|G|\geq 3$, $t \geq 1$, and $p \geq 1$,
\begin{align*}
    \diste(\mu^p(G)) \le \diste(G)
\end{align*}
and
\begin{align}
    \diste(\mu^p_t(G))\leq \diste(G). \nonumber
\end{align}
\end{cor}

\bibliographystyle{amsplain}
\bibliography{DisteMycielskiBib}

\providecommand{\bysame}{\leavevmode\hbox to3em{\hrulefill}\thinspace}
\providecommand{\MR}{\relax\ifhmode\unskip\space\fi MR }
\providecommand{\MRhref}[2]{%
  \href{http://www.ams.org/mathscinet-getitem?mr=#1}{#2}
}
\providecommand{\href}[2]{#2}
\begin{thebibliography}{10}

\bibitem{AC96}
Michael~O. Albertson and Karen~L. Collins, \emph{Symmetry breaking in graphs},
  Electron. J. Combin. \textbf{3} (1996), no.~1, Research Paper 18, approx. 17.
  \MR{1394549}

\bibitem{AS2020}
Saeid Alikhani and Samaneh Soltani, \emph{The distinguishing number and the
  distinguishing index of line and graphoidal graph(s)}, AKCE Int. J. Graphs
  Comb. \textbf{17} (2020), no.~1, 1--6. \MR{4145384}

\bibitem{B97}
L.~Babai, \emph{Asymmetric trees with two prescribed degrees}, Acta Math. Acad.
  Sci. Hungar. \textbf{29} (1977), no.~1-2, 193--200. \MR{453572}

\bibitem{BR2008}
R.~Balakrishnan and S.~Francis~Raj, \emph{Connectivity of the {M}ycielskian of
  a graph}, Discrete Math. \textbf{308} (2008), no.~12, 2607--2610.
  \MR{2410467}

\bibitem{BCKLPR22}
Debra Boutin, Sally Cockburn, Lauren Keough, Sarah Loeb, K.~E. Perry, and Puck
  Rombach, \emph{Distinguishing generalized {M}ycielskian graphs}, Australas.
  J. Combin. \textbf{83} (2022), 225--242. \MR{4435045}

\bibitem{BCKLPR24}
\bysame, \emph{Determining number and cost of generalized {M}ycielskian
  graphs}, Discuss. Math. Graph Theory \textbf{44} (2024), no.~1, 127--149.
  \MR{4686194}

\bibitem{ChXi2006}
Xue-gang Chen and Hua-ming Xing, \emph{Domination parameters in {M}ycielski
  graphs}, Util. Math. \textbf{71} (2006), 235--244. \MR{2278836}

\bibitem{FiMcBo1998}
David~C. Fisher, Patricia~A. McKenna, and Elizabeth~D. Boyer,
  \emph{Hamiltonicity, diameter, domination, packing, and biclique partitions
  of {M}ycielski's graphs}, Discrete Appl. Math. \textbf{84} (1998), no.~1-3,
  93--105. \MR{1626550}

\bibitem{KP15}
R~Kalinowski and M~Pil$\acute{\text{s}}$niak, \emph{Distinguishing graphs by
  edge-colourings}, European J. Combin. \textbf{45} (2015), 124--131.
  \MR{3286626}

\bibitem{LWLG2006}
Wensong Lin, Jianzhuan Wu, Peter Che~Bor Lam, and Guohua Gu, \emph{Several
  parameters of generalized {M}ycielskians}, Discrete Appl. Math. \textbf{154}
  (2006), no.~8, 1173--1182. \MR{2219419}

\bibitem{M55}
Jan Mycielski, \emph{Sur le coloriage des graphs}, Colloq. Math. \textbf{3}
  (1955), 161--162.

\bibitem{S85}
M.~Stiebitz, \emph{Beiträge zur theorie der färbungskritschen graphen}, Ph.D.
  thesis, Technische Universität Ilmenau, 1985.

\bibitem{SV2023}
Vinny Susan~Prebhath and V.~Sangeetha, \emph{Tadpole domination number of
  graphs}, Discrete Math. Algorithms Appl. \textbf{15} (2023), no.~8, Paper No.
  2250173, 12. \MR{4627856}

\end{thebibliography}

\end{document}